\documentclass{article}

\usepackage[utf8]{inputenc}
\usepackage{amsthm,amsmath,amsfonts,amssymb}
\usepackage{kotex}
\usepackage{tikz-cd}
\usepackage{comment}
\usepackage{MnSymbol}
\usepackage{graphicx}

\newcommand{\Mod}[1]{\ (\text{mod}\ #1)}

\newtheorem{lemma}{Lemma}[section]
\newtheorem{theorem}{Theorem}[section]

\newtheorem{prop}{Proposition}[section]

\newtheorem{definition}{Definition}[section]
\newtheorem{corollary}{Corollary}[section]

\newcommand{\zz}{\mathbb{Z}}

\newcommand{\ww}{\mathcal{W}}
\newcommand{\cc}{\mathbb{C}}

\newcommand{\mm}{\mathcal{M}}
\newcommand{\RR}{\mathcal{R}}
\newcommand{\hh}{\mathcal{H}}
\newcommand{\rr}{\mathbb{R}}
\newcommand{\qq}{\mathbb{Q}}
\newcommand{\ee}{\mathcal{E}}
\newcommand{\uu}{\mathcal{U}}
\newcommand{\SL}{\mathrm{SL}}
\newcommand{\GL}{\mathrm{GL}}

\newcommand{\QQ}{\mathcal{Q}}

\newcommand{\Man}{\mathrm{Man}}

\newcommand\m[1]{\begin{pmatrix}#1\end{pmatrix}}
\newcommand\sm[1]{\left(\begin{smallmatrix}#1\end{smallmatrix}\right)}

\title{Quantum modular forms and Hecke operators}
\author{See-woo Lee}
\begin{document}

\maketitle

\begin{abstract}
It is known that there are one-to-one correspondences among the space of cusp forms, the space of homogeneous period polynomials and the space of Dedekind symbols with polynomial reciprocity laws. We add one more space, the space of quantum modular forms with polynomial period functions, to extend results from Fukuhara. Also, we consider Hecke operators on the space of quantum modular forms and construct new quantum modular forms.
\end{abstract}

%\tableofcontents
\section{Introduction}

In \cite{shin98}, Fukuhara showed that there are correspondences among the space of cusp forms, Dedekind symbols and period polynomials. Also, he defined Hecke operators on the space of Dedekind symbols which are compatible with Hecke operators on other spaces. 
As an application,  famous congruences from Ramanujan such as $\tau(n)\equiv \sigma_{11}(n)\Mod{691}$  were rediscovered (For details, see \cite{shin06hec}).

In this paper, we extend Fukuhara's result by adding one more space $\QQ_{poly, -w}^{p\pm}$ on the diagram, the space of quantum modular forms with polynomial period functions.

\begin{theorem} Let $w\geq 2$ be an even integer. Define the following spaces : 
\begin{align*}
S_{w+2}&:=\text{a space of cusp forms of weight $w+2$ on $\mathrm{SL}_2(\zz)$} \\
\cc[h, k]_{w}&:=\text{a space of homogeneous polynomials of degree $w$} \\
\uu_{w}&:=\{g\in \cc[h, k]_{w}:g(h+k, k)+g(h, h+k)=g(h,k), g(1, 1)=0\} \\
\uu_{w}^{\pm}&:=\{g\in \uu_{w}:g(h,-k)=\pm g(h,k)\} \\
\ee_{w}&:=\{E:E \text{ is a Dedekind symbol of weight }w\text{ s.t. }E(h, k)-E(k, -h)\in \cc[h, k]_{w}\} \\
\ee_{w}^{\pm}&:=\{E\in \ee_{w}:E(h, -k)=\pm E(h, k)\text{ for all }(h, k)\in\zz^{+}\times\zz \}
\end{align*}
\begin{align*}
\QQ_{-w}&:=\text{a space of quantum modular forms of weight $-w$ on $\mathrm{SL}_{2}(\zz)$} \\
\QQ_{poly,-w}^{p}&:=\left\{\widetilde{f}\in \QQ_{-w}:\widetilde{f}(x)=\widetilde{f}(x+1), \widetilde{f}(x)-x^{w}\widetilde{f}\left(-\frac{1}{x}\right)=\sum_{j=0}^{w}c_{j}x^{j},c_{j}\in\cc\right\} \\
\QQ_{poly,-w}^{p\pm}&:=\{\widetilde{f}\in \QQ_{poly, -w}^{p}:\widetilde{f}(-x)=\pm \widetilde{f}(x)\} \\
\QQ_{\ee, -w}&:= \Psi_{w}(\ee_{w}) \\ 
\QQ_{\ee, -w}^{\pm}&:= \{ \widetilde{f}\in \QQ_{\ee, -w}\,:\, \widetilde{f}(-x) =\pm\widetilde{f}(x)\}
\end{align*}
\begin{enumerate}

\item The following diagram commutes.

\begin{center}
\begin{tikzcd}
& & S_{w+2} \arrow[dddd, "Q_{w+2}^{\pm}", pos=0.3] \arrow[lldd, "\alpha_{w+2}^{\pm}"] \arrow[rrdd, "R_{w+2}^{\pm}"] & & \\
& & & & \\
\ee_{w}^{\pm} \arrow[rrrr, "\beta_{w}^{\pm}", pos=0.7, crossing over] \arrow[rrdd, "\Psi_{w}^{\pm}"] & & & & \uu_{w}^{\pm} \\
& & & & \\
& & \QQ_{\ee, -w}^{\pm} \arrow[rruu, "H_{w}^{\pm}"] & & 
\end{tikzcd}
\end{center}
\item $\QQ_{\ee, -w}=\QQ_{poly, -w}^{p}$, $\QQ_{\ee, -w}^{\pm}=\QQ_{poly, -w}^{p\pm}$.
\item $\Psi_{w}^{-}:\ee_{w}^{-}\to \QQ_{\ee, -w}^{-}$, $Q_{w+2}^{-}:S_{w+2}\to \QQ_{\ee, -w}^{-}$, $H_{w}^{-}:\QQ_{\ee, -w}^{-}\to \uu_{w}^{-}$ are isomorphisms between vector spaces. 
\item $\Psi_{w}^{+}:\ee_{w}^{+}\to \QQ_{\ee, -w}^{+}$ is an isomorphism. 
\item $Q_{w+2}^{+}:S_{w+2}\to \QQ_{\ee, -w}^{+}$ is a monomorphism s.t. the image $Q_{w+2}^{+}(S_{w+2})$ is the subspace of $\QQ_{\ee, -w}^{+}$ of codimension $2$ where $Q_{w+2}^{+}(S_{w+2}), \Psi_{w}(F_{w}), \Psi_{w}(G_{w})$ span $\QQ_{\ee, -w}$. Here we have
$$
\Psi_{w}(F_{w})\left(\frac{k}{h}\right)\equiv 1, \qquad \Psi_{w}(G_{w})\left(\frac{k}{h}\right)=\left(\frac{\gcd(h, k)}{h}\right)^{w}.
$$
\item $H_{w}^{+}:\QQ_{\ee, -w}^{+}\to \uu_{w}^{+}$ is an epimorphism s.t. $\ker(H_{w}^{+})$ is one-dimensional subspace of $\QQ_{\ee, -w}^{+}$ generated by $\Psi_{w}(G_{w})$ 
\end{enumerate}
\end{theorem}

Definitions of the maps in the diagram are given in Section 2 and 3. 
Using these maps, we define the action of Hecke operators on the space of quantum modular forms which are compatible with other Hecke operators. 
Also, we extend this definition as Hecke operators on the  space of  quantum modular forms on a congruence subgroup with a nontrivial multiplier system. 
As a consequence, we construct quantum modular forms of weight 1 on $\Gamma_{0}(2)$ with nontrivial multiplier systems. 

\begin{theorem}
Let $f:\qq\to \cc$ be a weight 1 quantum modular form on $\Gamma_{0}(2)$ with nontrivial multiplier system $\chi$ defined by 
$$
\chi\left(\m{1&1\\0&1}\right)=\chi\left(\m{1&0\\2&1}\right)=\zeta_{24}
$$
where $\zeta_{24}=e^{2\pi i/24}$. 
For any prime $p$ with $5\leq p\leq 757$, define $T_{p}^{\infty}f:\qq\to \cc$ as
$$
T_{p}^{\infty}f(x)=(-1)^{\frac{p^{2}-1}{24}}f(px)+\frac{1}{p}\sum_{j=0}^{p-1}\zeta_{24}^{-pj}f\left(\frac{x+j}{p}\right)
$$
where $\zeta_{24}=e^{2\pi i /24}$.Then $T_{p}^{\infty}f$ is a quantum modular form of weight 1 on $\Gamma_{0}(2)$ with a multiplier system $\chi^{p}$. 
\end{theorem}

This paper is organized as follows. 
In Section 2, we review known facts. In Section 3 and 4, we give proofs of main results.

\emph{Acknowledgement}. This is part of the author's undergraduate thesis paper. 
The author is grateful to a advisor  Y. Choie for her helpful advice. 
The author is also grateful to S. Fukuhara, D. Choi and K. Ono for their comments via emails.

\section{Preliminaries}

%%%%%%%%%%%%%%%%%%% Modular forms %%%%%%%%%%%%%%%%

\subsection{Modular forms and Period polynomials}

(In this section, we will use homogenized version of period polynomials, which is slightly different from the notations in \cite{koza84}.) $\mathrm{SL}_{2}(\mathbb{Z})$ acts on $\hh=\{z\in \cc:\Im(z)>0\}$ via M\"obius transform and the natural boundary of $\hh$ is $\qq\cup\{\infty\}$, the set of cusps of $\mathrm{SL}_{2}(\zz)$. 
Let $f$ be a cusp form of weight $w+2$ on $\mathrm{SL}_{2}(\zz)$. 
The (homogenized) period polynomial $r_{f}(X, Y)$ associated to $f$ is defined by 
$$
r_{f}(X, Y):=\int_{0}^{i\infty}f(z)(Xz-Y)^{w}dz.
$$
Let $r^{+}_{f}$ (resp. $r^{-}_{f}$) be even (resp. odd) part of the period polynomial, i.e. 
$$
r_{f}^{\pm}(X, Y) = \frac{1}{2}(r_{f}(X, Y)\pm r_{f}(X, -Y))
$$
and $V_{w}=V_{w}(\cc)$ be the space of 2-variable homogeneous polynomials of degree $\leq w$ with coefficients in $\cc$. 
Then there is a right $\mathrm{SL}_{2}(\zz)$-action on $V_{w}(\cc)$ defined by 
$$
(P|\gamma)(X, Y)=(cX+dY)^{w}P\left(\frac{aX+bY}{cX+dY}\right)
$$
where $P(X, Y)\in V_{w}(\cc), \gamma=\left(\begin{smallmatrix}a&b\\c&d\end{smallmatrix}\right)$. 
We can naturally extend this action to the action of the group algebra $\qq[\mathrm{SL}_{2}(\zz)]$. 
Consider the subspace $W_{w}\subset V_{w}$, 
$$W_{w}=\ker(1+S)\cap\ker(1+U+U^{2})=\{P\in V_{w}:\,\,P+P|S=0,\,\,P+P|U+P|U^{2}=0\}$$
where $S=\left(\begin{smallmatrix}0&-1\\1&0\end{smallmatrix}\right), U=\left(\begin{smallmatrix}1&-1 \\ 1& 0 \end{smallmatrix}\right)$. 
We can express $W_{w}=W_{w}^{+}\oplus W_{w}^{-}$ where $W_{w}^{+}$ (resp. $W_{w}^{-}$) is the space of even (resp. odd) polynomials. 
 Using relations $S^{2}=-I$ and $U^{3}=-I$, one checks $r_{f}(X, Y)$ is in $W_{w}$. The following are well-known:
 
\begin{theorem}[Eichler-Shimura-Zagier]
The map $r^{-}:S_{w+2}\to W_{w}^{-}$ is an isomorphism. 
The map $r^{+}:S_{w+2}\to W_{w}^{+}$ is a monomorphism where $r^{+}(S_{w+2})$ is the subspace of $W_{w}^{+}$ of codimension 1, defined over $\qq$, and not containing the element $p_{0}(X, Y)=X^{w}-Y^{w}$. 
\end{theorem}
This plays an important role in the proof of Fukuhara's theorem.  
Choie and Zagier considered the action of Hecke operator on the space of period functions (the following theorem is the rephrased version of the Theorem 2 of \cite{choie93} - in the paper the authors used the notation $S = \sm{1& 1\\0&1}$ and $T = \sm{0&-1 \\ 1& 0}$ instead of $T = \sm{1& 1\\0&1}$ and $S = \sm{0&-1 \\ 1& 0}$. We will use the later one throughout this paper):

\begin{theorem}[Choie-Zagier]
\label{period}

Let $f$ be a cusp form of weight $w+2$ and $\mm_{n}:=\{M\in \mathrm{PSL}_{2}(\zz)\,|\,\det(M)=n\}$.
Then the period polynomial of $f|T_{n}=T_{n}f$ is given by 
$$
r_{T_{n}f}(x)=(\widetilde{T}_{n}r_{f})(x)=\sum_{\left(\begin{smallmatrix}a&b\\c&d\end{smallmatrix}\right)\in \Man_{n}}(cx+d)^{w}r_{f}\left(\frac{ax+b}{cx+d}\right).
$$ 
where 
\begin{align*}
\Man_{n}&=\left\{\begin{pmatrix}a&b \\c&d\end{pmatrix} \,\Bigr|\,a,b,c,d\in\zz, ad-bc=n,a>|c|, d>|b|, bc\leq 0,\right. \\
&\quad \left. b=0\Rightarrow -\frac{a}{2}<c\leq \frac{a}{2}, c=0\Rightarrow -\frac{d}{2}<b\leq\frac{d}{2}\right\}.
\end{align*}

Also, we can consider $\widetilde{T}_{n}$ as a Hecke operator on the space of period functions. 
If we regard $\widetilde{T}_{n}$ as an element in $\RR_{n}=\zz[\mm_{n}]$ with $\SL_{2}(\zz)$-action via multiplication (these elements have a right action on the space of cusp forms or period functions via slash operator), then there exists $X_{n}, Y_{n} \in \RR_{n}$ s.t. 
$$
T_{n}^{\infty}(I-T)=(I-T)X_{n},\qquad T_{n}^{\infty}(I-S)=(I-S)\widetilde{T}_{n}+(I-T)Y_{n}.
$$
where 
$$
T_{n}^{\infty} = \sum_{\substack{ad=n, a, d>0 \\ b\Mod{d}}} \begin{pmatrix} a&b \\ 0 & d \end{pmatrix} \in ((T-I)\RR_{n})\backslash \RR_{n}. 
$$
\end{theorem}

%%%%%%%%%%%%% Dedekind symbols %%%%%%%%%%%%%

\subsection{Dedekind symbols and correspondences}

We follow the definitions and notations given in \cite{shin98} and \cite{shin06hec}. 
For any positive even number $w$, weight $w$ Dedekind symbols are functions $E^{\epsilon}:\zz^{+}\times \zz\to \cc$ satisfying

\begin{enumerate}
\item $E^{\epsilon}(h,k)=E^{\epsilon}(h, k+h)$
\item $E^{\epsilon}(h, -k)=\epsilon E^{\epsilon}(h, k)$
\item $E^{\epsilon}(h, k)-E^{\epsilon}(k, -h)=f^{\epsilon}(h, k)$
\item $E^{\epsilon}(ch, ck)=c^{w}E^{\epsilon}(h, k)$
\end{enumerate}
for some $f:\zz^{+}\times\zz^{+}\to \cc$ and $\epsilon\in \{\pm\}$. The function $f^{\epsilon}$ is called the reciprocity function of $E^{\epsilon}$. Note that $E^{+}$ (resp. $E^{-}$) is called even (resp. odd) Dedekind symbol. 
In \cite{shin98}, it was shown that one can recover Dedekind symbol $E^{\epsilon}$ of weight $w$ from any function $f^{\epsilon}$ with reciprocity properties, up to constant multiple of $G_{w}(h, k)=\gcd(h, k)^{w}.$

\begin{comment}
\begin{definition}
Let $w$ be a positive even number and $\cc[h,k]_{w}$ be the space of homogeneous polynomial of degree $w$. 
We use the following notations: 
\begin{align*}
\uu_{w}&:=\{g\in \cc[h, k]_{w}:g(h+k, k)+g(h, h+k)=g(h,k), g(1, 1)=0\} \\
\uu_{w}^{-}&:=\{g\in \uu_{w}:g(h,-k)=-g(h,k)\} \\
\uu_{w}^{+}&:=\{g\in \uu_{w}:g(h, -k)=g(h,k)\} \\
\ee_{w}&:=\{E:E \text{ is a Dedekind symbol of weight }w\text{ s.t. }E(h, k)-E(k, -h)\in \cc[h, k]_{w}\} \\
\ee_{w}^{-}&:=\ee_{w}\cap\ww_{w}^{+}\\
\ee_{w}^{+}&:=\ee_{w}\cap \ww_{w}^{-} \\
\end{align*}
\end{definition}
\end{comment}
%%Obviously,  $\ee_{w}=\ee_{w}^{+}\oplus\ee_{w}^{-}$ and $\uu_{w}=\uu_{w}^{+}\oplus\uu_{w}^{-}$. 
For  $f\in S_{w+2}$ and $(h, k)\in \zz^{+}\times \zz$, define $E_{f}$ by $$ E_{f}(h, k)=\int_{k/h}^{i\infty}f(z)(hz-k)^{w}dz.$$
Furthermore we define $E_{f}^{-}$ and $E_{f}^{+}$, respectively, by
$$
E_{f}^{\pm}(h, k)=\frac{1}{2}(E_{f}(h, k)\pm E_{f}(h, -k)).
$$
Then it is shown that $E_{f}$ is a Dedekind symbol of weight $w$ with polynomial reciprocity function, so this defines the map $\alpha_{w+2}^{\pm}:S_{w+2}\to \ee_{w}^{\pm}$
with $\alpha_{w+2}^{\pm}(f)=E_{f}^{\pm}$. 

Also, define a map $\beta_{w}:\ee_{w}\to \uu_{w}$ which sends Dedekind symbol to its reciprocity function, 
\begin{align*}
\beta_{w}(E)(h, k)&:=E(h, k)-E(k, -h)\\
\beta_{w}^{\pm}(E)(h, k)&:=\frac{1}{2}(\beta_{w}(E)(h, k)\pm \beta_{w}(E)(h, -k)).
\end{align*}
 In case of Dedekind symbol $E_{f}$ associate with a cusp form $f$, we can check that
\begin{align*}
\beta_{w}(E_{f})(h, k)=\int_{0}^{i\infty}f(z)(hz-k)^{w}dz
\end{align*}
where RHS is a homogenized period polynomial of $f$. We will denote this map as $R_{w+2}$ and $R_{w+2}^{\pm}:=\beta_{w}^{\pm}\circ \alpha_{w+2}^{\pm}$, similarly.

Fukuhara showed that there is a one-to-one correspondence among the spaces $S_{w+2}, \ee_{w}^{\pm}$ and $\uu_{w}^{\pm}$  (see \cite{shin07}, \cite{shin06hec},  \cite{shin98}).

\begin{theorem}[Fukuhara]
The following diagram commutes : 
\begin{center}
\begin{tikzcd}
& & S_{w+2} \arrow[lldd, "\alpha_{w+2}^{\pm}"] \arrow[rrdd, "R_{w+2}^{\pm}"] & & \\
& & & & \\
\ee_{w}^{\pm} \arrow[rrrr, "\beta_{w}^{\pm}", pos=0.7, crossing over] & & & & \uu_{w}^{\pm} \end{tikzcd}
\end{center}
$\alpha_{w+2}^{-}:S_{w+2}\to \ee_{w}^{-}$ is an isomorphism and $\alpha_{w+2}^{+}:S_{w+2}\to \ee_{w}^{+}$ is a monomorphism s.t. the image $\alpha_{w+2}^{+}(S_{w+2})$ is the subspace of $\ee_{w}^{+}$ of codimension 1 where $\alpha_{w+2}^{+}(S_{w+2}), G_{w}$ span $\ee_{w}^{+}$.
Also, $\beta_{w}^{-}:\ee_{w}^{-}\to \uu_{w}^{-}$ is an isomorphism, and $\beta_{w}^{+}:\ee_{w}^{+}\to \uu_{w}^{+}$ is an epimorphism s.t. $\ker\beta_{w}^{+}$ is one dimensional subspace of $\ee_{w}^{+}$ spanned by $G_{w}$. 
\end{theorem}

In \cite{shin06hec}, Fukuhara defined Hecke operators $T_{n}^{\infty}$ on the space of Dedekind symbols which are compatible with those on the space of  modular forms. 

\begin{theorem}[Fukuhara]

Let $E\in\ee_{w}$ be a weight $w$ Dedekind symbol. Then Hecke operators $T_{n}^{\infty}$ on $\ee_{w}$ which are defined by
$$
(T_{n}^{\infty}E)(h, k):=\sum_{\gamma\in\Gamma_{1}\backslash M_{n}}(E|\gamma)(h, k)=\sum_{ad=n, d>0}\sum_{b=0}^{d-1}E(dh, ak+bh)
$$
preserves $\ee_{w}^{\pm}$ and compatible with Hecke operators on $S_{w+2}$, i.e. 
$$
\alpha_{w+2}(T_{n}f) = T_{n}^{\infty} \alpha_{w+2}(f)
$$
for any $f\in S_{w+2}$. 

\end{theorem}

%%%%%%%%%%% Quantum modular forms %%%%%%%%%%%%%

\subsection{Quantum modular forms}

Quantum modular forms were first defined by Zagier in his paper \cite{za10}. 
They are functions defined on $\qq$ with modular properties, which are slightly different from usual modular forms. 
Recall the definition of quantum modular forms in \cite{choi16}.

\begin{definition}
Let $N$ be a positive integer, $k\in \frac{1}{2}\zz$ and $\chi$ be a multiplier system on $\Gamma_{0}(N)$. 
Then a function $f:\qq\to\cc$ is a quantum modular form of weight $k$ and a multiplier system $\chi$ on $\Gamma_{0}(N)$  if it satisfies the modular relation
$$
f(x)-(f|_{k,\chi}\gamma)(x)=h_{\gamma}(x)
$$
for all $\gamma\in \Gamma_{0}(N)$ where $$ (f|_{k, \chi}\gamma)(x)=\chi(\gamma)^{-1}(cx+d)^{-k}f\left(\frac{ax+b}{cx+d}\right), \quad \gamma=\begin{pmatrix} a&b \\ c&d\end{pmatrix}\in \Gamma_{0}(N)$$ and $h_{\gamma}$  can be extended smoothly on $\rr$ except finitely many points $S\subset \qq$. 
Let $\QQ_{k}(\Gamma_{0}(N), \chi)$ be the space of weight $k$ quantum modular forms on $\Gamma_{0}(N)$. 
Also, we denote $\QQ_{k}(\SL_{2}(\zz), id):=\QQ_{k}$ for short.
\end{definition}
$h_{\gamma}$ is a 1-cocycle, i.e. $h_{\gamma_{1}\gamma_{2}}=h_{\gamma_{1}}|_{k, \chi}\gamma_{2}+h_{\gamma_{2}}$. 
Let $\QQ_{k}^{p}\subset \QQ_{k}$ be the subspace containing quantum modular forms of weight $k$ with $h_{T}(x)\equiv 0$, i.e. $f(x+1)=f(x)$. 
Then $h_{U}=h_{TS}=h_{S}+h_{T}|_{k}S=h_{S}$, so $h_{S}$ became a period function. 
For $f\in \QQ_{k}^{p}$, we call $h_{S}(x)=f(x)-x^{-k}f\left(-\frac{1}{x}\right)$ as the \emph{period function of $f$}. 

For example, Zagier found that we can associate a quantum modular form to a certain Maass wave form. More precisely, let's recall one of the $q$-hypergeometric function from Ramanujan's "Lost" Notebook : 
$$
\sigma(q)=\sum_{n=0}^{\infty}\frac{q^{n(n+1)/2}}{(1+q)(1+q^{2})\cdots(1+q^{n})}=1+q-q^{2}+2q^{3}-2q^{4}+q^{5}+q^{7}+\cdots.
$$
Now define coefficients $\{T(n)\}_{n\in 24\zz+1}$ by $q\sigma(q^{24})=\sum_{n\geq 0}T(n)q^{n}$.    In \cite{co88}, Cohen showed that these are related to the certain Maass wave form given by 
$$
u(z)=\sqrt{y}\sum_{n\in 24\zz+1}T(n)K_{0}(2\pi |n|y/24)e^{2\pi i n x/24}, \quad z = x+iy
$$
which is a Maass wave form of eigenvalue $1/4$ on a congruence subgroup $\Gamma_{0}(2)$. Also, in \cite{an86} Andrews proved the following $q$-series identity 
$$
\sigma(q)=1+\sum_{n=0}^{\infty}(-1)^{n}q^{n+1}(1-q)(1-q^{2})\cdots(1-q^{n}). 
$$
This implies that $\sigma(q)$ also makes sense whenever $q$ is a root of unity because the series only contains finite sum in that case. Now we can define $f:\qq\to \cc$ as $$f(x)=q^{1/24}\sigma(q)\,\,\,\,(x\in \qq, q=e^{2\pi i x}).$$ Zagier proved that this function satisfies \emph{quantum} modular properties : 
\begin{prop}
The above function $f$ satisfies 
$$
f(x+1)=\zeta_{24}f(x),\,\,\,\,\frac{1}{|2x+1|}f\left(\frac{x}{2x+1}\right)=\zeta_{24}f(x)+h(x)
$$
where $\zeta_{24}=e^{2\pi i /24}$ and $h:\rr\to \cc$ is $C^{\infty}$ on $\rr$ and real-analytic except at $x=-1/2$. 
\end{prop}
In this case, the slash operator is given by 
$$
(f|\gamma)(x)=\chi(\gamma)^{-1}|cx+d|^{-k}f\left(\frac{ax+b}{cx+d}\right), \qquad\gamma=\m{a&b\\c&d}\in \GL_{2}^{+}(\qq).
$$
This is a weight 1 quantum modular form on $\Gamma_{0}(2)$ with multiplier system $\chi$ defined as
$\chi(T) = \chi(R) = \zeta_{24}$ where $$
T = \begin{pmatrix} 1&1\\0&1\end{pmatrix}, \qquad R = \begin{pmatrix} 1 & 0 \\ 2& 1 \end{pmatrix}. 
$$
Note that $T, R$ generates $\Gamma_{0}(2)$. From $-I = (RT^{-1})^{2}$, we have $\chi(-I) = 1$. 

\emph{Remark.} In \cite{za10}, there is a minor errata : we have to take an absolute value on $2x+1$. 

%%%%%%%%%%%%%%% Result %%%%%%%%%%%%%%%%

\section{Quantum modular forms associated to Dedekind symbols with polynomial reciprocity law}

Let $E\in \ee_{w}$ be a Dedekind symbol of weight $w$ with polynomial reciprocity law where $w$ is a positive even integer. 
We can define a quantum modular form $\widetilde{f}_{E}$ associated to $E$. 

\begin{definition}
Let $E\in \ee_{w}$. Define a map $\Psi_{w}:\ee_{w}\to \QQ_{-w}$ by $\Psi_{w}(E)=\widetilde{f}_{E}$ where 
$$
\widetilde{f}_{E}\left(\frac{k}{h}\right)=h^{-w}E(h, k).
$$ 
\end{definition}
It is clear that $\Psi_{w}$ is injective map. 

\begin{prop}
$\widetilde{f}_{E}$ is a well-defined quantum modular form of weight $-w$ with polynomial period function on $\mathrm{SL}_{2}(\mathbb{Z})$. 
\end{prop}

\begin{proof}
Since $E$ is a Dedekind symbol of weight $w$, $\widetilde{f}_{E}$ is well-defined. 
For $E\in \ee_{w}$, 
$$ E(h, k)-E(k, -h)=\sum_{i=0}^{w}\alpha_{i}h^{i}k^{w-i}$$ for some $\alpha_{i}\in \mathbb{C}$. 
Then for $x=k/h\in \qq$, we have
$$\widetilde{f}_{E}(x+1)=\frac{1}{h^{w}}E(h, h+k)=\frac{1}{h^{w}}E(h, k)=\widetilde{f}_{E}(x)$$
and
\begin{align*}
\widetilde{f}_{E}(x)-x^{w}\widetilde{f}_{E}\left(-\frac{1}{x}\right)&=\frac{1}{h^{w}}E(h, k)-\left(\frac{k}{h}\right)^{w}\frac{1}{k^{w}}E(k, -h)=\frac{1}{h^{w}}(E(h, k)-E(k, -h))\\
&=\frac{1}{h^{w}}R(h, k)=\sum_{i=0}^{w}\alpha_{i}\left(\frac{k}{h}\right)^{w-i}=\sum_{i=0}^{w}\alpha_{i}x^{w-i}.
\end{align*}
Thus $\widetilde{f}_{E}$ is a quantum modular form of weight $-w$ with a trivial multiplier system. \qed
\end{proof}

Let  $\QQ_{poly,-w}^{p}$ be the space of periodic quantum modular forms of weight $-w$  with the trivial multiplier system and a polynomial period function, i.e. 
\begin{align*}
\QQ_{poly, -w}^{p}:=\left\{\widetilde{f}\in \QQ_{-w}:\widetilde{f}(x)=\widetilde{f}(x+1), \widetilde{f}(x)-x^{w}\widetilde{f}\left(-\frac{1}{x}\right)=\sum_{j=0}^{w}c_{j}x^{j},c_{j}\in\cc\right\}
\end{align*}
 and $\QQ_{\ee, -w}:=\Psi_{w}(\ee_{w})$ be image of $\Psi_{w}$. Note that $\QQ_{\ee, -w}\subseteq \QQ_{poly, -w}^{p}$. 
By composing this map with $\alpha_{w+2}$, consider  $$Q_{w+2}:=\Psi_{w}\circ\alpha_{w+2}:S_{w+2}\to \QQ_{\ee, -w}.$$

Note that $Q_{w+2}(f)$ coincides with the Eichler integral, 
$$
Q_{w+2}(f)(x)=\int_{x}^{i\infty}f(z)(x-z)^{w}dz=\frac{w!}{(-2\pi i)^{w+1}}\sum_{n\geq 1}\frac{a_{n}}{n^{w+1}}e^{2\pi i n x}
$$
and its period function is same as the (1-variable) period polynomial $r_{f}(x)$.

Also, we define odd and even part of the map $\Psi_{w}$, 
\begin{align*}
\Psi_{w}^{\pm}(E)(h, k):=&\frac{1}{2}(\Psi_{w}(E)(h, k)\pm \Psi_{w}(E)(h, -k)) \\
Q_{w+2}^{\pm}:=&\Psi_{w}^{\pm}\alpha_{w+2}^{\pm}:S_{w+2}\to \QQ_{\ee, -w}^{\pm}
\end{align*}
where
\begin{align*}
\QQ_{\ee, -w}^{+}&=\{\widetilde{f}_{Q}\in \QQ_{\ee, -w}:\,\,\widetilde{f}_{Q}(x)=\widetilde{f}_{Q}(-x)\} \\
\QQ_{\ee, -w}^{-}&=\{\widetilde{f}_{Q}\in \QQ_{\ee, -w}:\,\,\widetilde{f}_{Q}(x)=-\widetilde{f}_{Q}(-x)\}.
\end{align*}

Now we will define  new maps $H_{w}, H_{w}^{\pm}:\QQ_{poly, -w}^{p}\to \uu_{w}$ and we will prove the Theorem 1.
\begin{definition}
For any $\widetilde{f}\in \QQ_{poly, -w}^{p}$, define
\begin{align*}
H_{w}(\widetilde{f})(h, k)&:=h^{w}\widetilde{f}\left(\frac{k}{h}\right)-k^{w}\widetilde{f}\left(-\frac{h}{k}\right) \\
H_{w}^{\pm}(\widetilde{f})(h, k)&:=\frac{1}{2}(H_{w}^{\pm}(\widetilde{f})(h, k)\pm H_{w}^{\pm}(\widetilde{f})(h, -k)).
\end{align*}
\end{definition}

\begin{proof}[Proof of the Theorem 1]
First, we know that $g(x)=h_{S}(x)=\widetilde{f}(x)-x^{w}\widetilde{f}\left(-\frac{1}{x}\right)$ is a polynomial in $x$. 
Then $H_{w}(\widetilde{f})(h, k)=h^{w}g\left(\frac{k}{h}\right)$ is homogeneous polynomial in $h, k$ which satisfies period relations, $H_{w}(f)+H_{w}(f)|S=H_{w}(f)+H_{w}(f)|U+H_{w}(f)|U^{2}=0$. 

For $1$, it is enough to check that $\beta_{w}=H_{w}\Psi_{w}$. (If we show this, then $H_{w}Q_{w+2}=H_{w}\Psi_{w}\alpha_{w+2}=\beta_{w}\alpha_{w+2}=R_{w+2}$, so everything commutes.) 
It can be shown by direct computation : for any $E\in \ee_{w}$, 
\begin{align*}
H_{w}\Psi_{w}(E)(h, k)&=H_{w}(h^{-w}E(h, k))=h^{w}h^{-w}E(h, k)-k^{w}k^{-w}E(k, -h)\\
&=E(h, k)-E(k, -h)=\beta_{w}(E)(h, k),
\end{align*}
so $H_{w}\Psi_{w}=\beta_{w}$. 

%%%%%%%%%%%%%%%% FIX THIS %%%%%%%%%%%%%%%%
To show $2$, first we will check that $H_{w}^{-}$ is injective. 
This directly follows from the following Lemma :
\begin{lemma}
\label{lem}
Let $\widetilde{f}:\qq\to\cc$ be a function satisfies 
\begin{align*}
\widetilde{f}(x)-\widetilde{f}(x+1)=0 \quad\dots\quad(1)\\
\widetilde{f}(x)-x^{w}\widetilde{f}\left(-\frac{1}{x}\right)=0\quad\dots\quad(2)
\end{align*}
for some even integer $w\geq 2$. Then there exists $c\in \cc$ s.t. $$\widetilde{f}\left(\frac{k}{h}\right)=c\left(\frac{\gcd(h, k)}{h}\right)^{w}\quad\dots\quad(*)$$ for any $h, k\in \zz, h> 0$. 
In particular, if $\widetilde{f}$ is odd function then $\widetilde{f}\equiv 0$. 
\end{lemma}

\begin{proof}
We use induction on $h$ where $x=k/h$. Let $\widetilde{f}(0)=c$. 
For any $k\in \zz$, since $\widetilde{f}(k)=\widetilde{f}(0)=c$ by (1) so we proved $(*)$ for $h=1$. 
Suppose $(*)$ holds for any $k/h\in \qq$ with $h\leq h_{0}$ and consider $x=k/(h_{0}+1)$. 
By (1), we can assume $0\leq k\leq h_{0}$. 
If $k=0$, $\widetilde{f}(x)=\widetilde{f}(0)=c$ and we are done. 
If not, by (2)
$$
\widetilde{f}\left(\frac{k}{h_{0}+1}\right)=\left(\frac{k}{h_{0}+1}\right)^{w}\widetilde{f}\left(-\frac{h_{0}+1}{k}\right)
$$
and induction hypothesis gives
$$
\widetilde{f}\left(\frac{k}{h_{0}+1}\right)=\left(\frac{k}{h_{0}+1}\right)^{w}c\left(\frac{\gcd(k, h_{0}+1)}{k}\right)^{w}=c\left(\frac{\gcd(h_{0}+1, k)}{h_{0}+1}\right)^{w}.
$$
So $(*)$ holds for any $x\in \qq$. If $\widetilde{f}$ is odd function, $\widetilde{f}(-1)=\widetilde{f}(0)=\widetilde{f}(1)=-\widetilde{f}(-1)$, so $c=\widetilde{f}(0)=\widetilde{f}(-1)=0$ and $\widetilde{f}\equiv 0$. \qed
\end{proof}
Note that $\widetilde{f}\in \ker H_{w}^{-}$ exactly satisfies equations $(1)$ and $(2)$ in the Lemma \ref{lem}. 
Now suppose $\widetilde{f}\in \QQ_{poly, -w}^{p-}$, then $H_{w}^{-}(\widetilde{f})\in \uu_{w}^{-}$ so by Eichler-Shimura theory there exists $f\in S_{w+2}$ s.t. $R_{w+2}^{-}(f)=H_{w}^{-}(\widetilde{f})$. 
However, we have $R_{w+2}^{-}=H_{w}^{-}\Psi_{w}^{-}\alpha_{w+2}^{-}$ and injectivity of $H_{w}^{-}$ gives $\widetilde{f}=\Psi_{w}^{-}\alpha_{w+2}^{-}(f)$. So $\QQ_{poly, -w}^{p-}=\QQ_{\ee, -w}^{-}$. 

Then even case is similar. By the Lemma \ref{lem}, we have $\ker(H_{w}^{+})=\{c\Psi_{w}(G_{w}):c\in\cc\}$. For $\widetilde{f}\in\QQ_{poly, -w}^{+}$, $H_{w}^{+}(\widetilde{f})\in \uu_{w}^{+}$ and  Eichler-Shimura theory gives that 
$$
H_{w}^{+}(\widetilde{f})=a(h^{w}-k^{w})+R_{w+2}^{+}(f)=aH_{w}^{+}\Psi_{w}^{+}(F_{w})+R_{w+2}^{+}(f)
$$
for some $f\in S_{w+2}$ and $a\in \cc$. Since $R_{w+2}^{+}=H_{w}^{+}\Psi_{w}^{+}\alpha_{w+2}^{+}$ we get 
$$
H_{w}^{+}(\widetilde{f}-\Psi_{w}^{+}\alpha_{w+2}^{+}(f)-a\Psi_{w}^{+}(F_{w}))=0
$$
and
$$
\widetilde{f}-\Psi_{w}^{+}\alpha_{w+2}^{+}(f)-a\Psi_{w}^{+}(F_{w})=b\Psi_{w}^{+}(G_{w})
$$
which proves $\widetilde{f}\in \QQ_{\ee, -w}^{+}$, $\QQ_{poly, -w}^{p+}=\QQ_{\ee, -w}^{+}$. 
$\QQ_{\ee, -w}=\QQ_{poly, -w}^{p}$ immediately follows from $\QQ_{\ee, -w}=\QQ_{\ee, -w}^{+}\oplus \QQ_{\ee, -w}^{-}$ and $\QQ_{poly, -w}^{p}=\QQ_{poly, -w}^{p+}\oplus \QQ_{poly, -w}^{p-}$. 
%%%%%%%%%%%%%%%%%%%%%%%%%%%%%%%%%%%%%%

We know that $\Psi_{w}:\ee_{w}\to \QQ_{\ee, -w}$ is an isomorphism and $\Psi_{w}(\ee_{w}^{\pm})\subseteq \QQ_{\ee, -w}^{\pm}, \ee_{w}=\ee_{w}^{+}\oplus \ee_{w}^{-}, \QQ_{\ee, -w}=\QQ_{\ee, -w}^{+}\oplus \QQ_{\ee, -w}^{-}$, so we get $\Psi_{w}(\ee_{w}^{\pm})=\QQ_{\ee, -w}^{\pm}$ and $\Psi_{w}$ is an isomorphism between $\ee_{w}^{\pm}$ and $\QQ_{\ee, -w}^{\pm}$. 
Combining with the Theorem 5, we get the 3, 4, 5, 6. \qed
\end{proof}

\section{Hecke Operators on Quantum modular forms}

Now we define a Hecke operator on $\QQ_{\ee, -w}$ and show that it can  extends to  the space $\QQ_{-w}$. 
Actually, it is the same as an operator $T_{n}^{\infty}$ in the Theorem \ref{period}. 
\begin{definition} 
Define a Hecke operator $T_{n}^{\infty}$ on $\QQ_{\ee, -w}$ by 
$$
(T_{n}^{\infty}\widetilde{f})(x):=\sum_{ad=n, d>0}d^{w}\sum_{b=0}^{d-1}\widetilde{f}\left(\frac{ax+b}{d}\right).
$$
\end{definition}
\begin{prop}
\begin{enumerate}
\item The Hecke operator on $\QQ_{\ee,-w}$ is compatible with Hecke operator on $\ee_{w}$, i.e. the following diagram commutes
\begin{center}
\begin{tikzcd}
\ee_{w} \arrow[r, "T_{n}^{\infty}"] \arrow[d, "\Psi_{w}"] & \ee_{w}\arrow[d, "\Psi_{w}"] \\
\QQ_{\ee, -w}  \arrow[r, "T_{n}^{\infty}"] & \QQ_{\ee, -w} 
\end{tikzcd}
\end{center}
\item For any $f\in \QQ_{\ee, -w}, T_{n}^{\infty}f\in \QQ_{\ee, -w}$. 
\end{enumerate}
\end{prop}
\begin{proof}
Choose any $E\in \ee_{w}$ then we have to show $T_{n}^{\infty}\widetilde{f}_{E}=\widetilde{f}_{T_{n}^{\infty}E}$. By the direct computation, 
\begin{align*}
T_{n}^{\infty}\widetilde{f}_{E}\left(\frac{k}{h}\right)&=\sum_{ad=n, d>0}d^{w}\sum_{b=0}^{d-1}\widetilde{f}_{E}\left(\frac{ak+bh}{dh}\right) \\
&=\sum_{ad=n, d>0}d^{w}\sum_{b=0}^{d-1}(dh)^{-w}E(dh, ak+bh) \\
&=h^{-w}\sum_{ad=n, d>0}\sum_{b=0}^{d-1}E(dh, ak+bh) \\
&=h^{-w}(T_{n}^{\infty}E)(h, k)=\widetilde{f}_{T_{n}^{\infty}E}\left(\frac{k}{h}\right)
\end{align*}
for any $(h, k)\in\zz^{+}\times\zz$ so we get $T_{n}^{\infty}\widetilde{f}_{E}=\widetilde{f}_{T_{n}^{\infty}E}$. 
 $T_{n}^{\infty}(\QQ_{\ee, -w})\subseteq \QQ_{\ee, -w}$ easily follows from compatibility. \qed
\end{proof}

Since Hecke operators on $S_{w+2}$ are compatible with Hecke operators on $\ee_{w}$, we get the following corollary. 
\begin{corollary}
The Hecke operator on  $\QQ_{\ee, -w}$ is compatible with Hecke operator on $S_{w+2}$, i.e. the following diagram commutes
\begin{center}
\begin{tikzcd}
S_{w+2} \arrow[r, "T_{n}"] \arrow[d, "Q_{w+2}"] & S_{w+2} \arrow[d, "Q_{w+2}"] \\
\QQ_{\ee, -w}  \arrow[r, "T_{n}^{\infty}"] & \QQ_{\ee, -w} 
\end{tikzcd}
\end{center}
\end{corollary}

Now we will show that the definition of Hecke operator on $\QQ_{\ee, -w}$ can be extended to the whole space $\QQ_{-w}$. 
\begin{theorem}
$T_{n}^{\infty}$ can be extended to $\QQ_{-w}$ by the same way : for any quantum modular form $\widetilde{f}\in \QQ_{-w}$, define
$$
(T_{n}^{\infty}\widetilde{f})(x):=\sum_{ad=n, d>0}d^{w}\sum_{b=0}^{d-1}\widetilde{f}\left(\frac{ax+b}{d}\right)
$$
then $T_{n}^{\infty}\widetilde{f}\in \QQ_{-w}$. Also, $T_{n}^{\infty}(\QQ_{-w}^{p})\subseteq \QQ_{-w}^{p}, T_{n}^{\infty}(\QQ_{poly, -w}^{p})\subseteq \QQ_{poly, -w}^{p}$. 
\end{theorem}

\begin{proof}
We only need to check $(f|T_{n}^{\infty}(I-T))(x)$ and $(f|T_{n}^{\infty}(I-S))(x)$ can be regarded as smooth function on $\rr$ with finitely many singular points. 
Using the Theorem \ref{period} again, we have
\begin{align*}
(f|T_{n}^{\infty}(I-T))(x) &=(f|(I-T)X_{n})(x)=(h_{T}|X_{n})(x) \\
(f|T_{n}^{\infty}(I-S))(x)&=(f|(I-S)\widetilde{T}_{n})(x)+(f|(I-T)Y_{n})(x)\\
&=\widetilde{T}_{n}h_{S}(x)+(h_{T}|Y_{n})(x)
\end{align*}
and all of these can be extended smoothly on $\rr$ except finitely many points since $h_{T}$ and $h_{S}$ does.  
If $f\in \QQ_{-w}^{p}$, then period function of $T_{n}^{\infty}f$ is $\widetilde{T}_{n}h$ where $h$ is a period function of $f$. \qed
\end{proof}

In fact, almost all quantum modular forms are defined on a congruence subgroup with a nontrivial multiplier system, so our definition of the Hecke operator is not very useful. 
So we will define some sort of generalized version of the Hecke operator; which are defined on the space of quantum modular forms of integer weight with a nontrivial multiplier system $\chi$ on some congruence subgroup. 
We will check that this operator \emph{changes} multiplier system.

\begin{definition}
Let $\Gamma\leq \SL_{2}(\zz)$ be a congruence subgroup and let $\chi, \chi':\Gamma\to \mathbb{S}^{1}$ be any two multiplier systems and $\alpha\in \GL_{2}^{+}(\qq)$. 
For any $\alpha\in \GL_{2}^{+}(\qq)$, we call two multiplier systems are \emph{compatible at $\alpha$} if the function $c_{\chi, \chi'}:\Gamma\alpha\Gamma\to \mathbb{S}^{1}$ defined by
$$
c_{\chi, \chi'}(\gamma_{1}\alpha\gamma_{2})=\chi(\gamma_{1})\chi'(\gamma_{2})
$$ 
is a well-defined function, i.e. for any element $\gamma_{1}\alpha\gamma_{2}=\delta_{1}\alpha\delta_{2}$ in $\Gamma\alpha\Gamma$, we have
$$
\chi(\gamma_{1})\chi'(\gamma_{2})=\chi(\delta_{1})\chi'(\delta_{2}).
$$
\end{definition}

We can easily check that for any given multiplier system $\chi$ and matrix $\alpha$, there is at most one multiplier system $\chi'$ compatible at $\alpha$.

Using this function, we can define a Hecke operator on the space of quantum modular forms which changes multiplier system. 
\begin{theorem}
Let $\alpha\in \GL_{2}^{+}(\qq)$ and $\Gamma\leq\SL_{2}(\zz)$ be a congruence subgroup. 
Let $\{\beta_{j}=\alpha_{j}\alpha\alpha_{j}'\}_{j\in J}$ be the set of representatives of orbits $\Gamma\backslash \Gamma\alpha\Gamma$. 
Suppose two multiplier systems $\chi, \chi':\Gamma\to\mathbb{S}^{1}$ are compatible at $\alpha$, i.e. there exists well-defined function $c_{\chi, \chi'}:\Gamma\alpha\Gamma\to \mathbb{S}^{1}$. 
Then for $f\in \QQ_{k}(\Gamma, \chi)$, define a Hecke operator $T^{\infty}_{\alpha, \chi, \chi'}$ by
$$
T^{\infty}_{\alpha, \chi, \chi'}f=\sum_{j}c_{\chi, \chi'}(\beta_{j})^{-1}f|\beta_{j}
$$
where $|$ is $k$-th slash operator which satisfies $f|\gamma_{1}\gamma_{2}=(f|\gamma_{1})|\gamma_{2}$ for any $\gamma_{1}, \gamma_{2}\in \GL_{2}^{+}(\qq)$. 
Then $T_{\alpha, \chi, \chi'}f\in \QQ_{k}(\Gamma, \chi')$.
\end{theorem}

\begin{proof}
For $\gamma\in \Gamma$, there exists a permutation $\sigma_{\gamma}:J\to J$ such that for each $j\in J$, $\beta_{j}\gamma\in \Gamma\beta_{\sigma_{\gamma}(j)}$, i.e.
$
\beta_{j}\gamma=\gamma_{j}\beta_{\sigma_{\gamma}(j)}
$
for some $\gamma_{j}\in \Gamma$.
Since $f\in \QQ_{k}(\Gamma, \chi)$, for any $\gamma\in \Gamma$,
$$
f-\chi(\gamma)^{-1}f|\gamma=h_{\gamma}
$$
where $h_{\gamma}$ is a function on $\qq$ which can be extended to $\rr$ smoothly except finitely many points. Then 
\begin{align*}
&T_{\alpha, \chi, \chi'}^{\infty}f-\chi'(\gamma)^{-1}T^{\infty}_{\alpha, \chi, \chi'}f|\gamma \\
&=\sum_{j}c_{\chi, \chi'}(\beta_{j})^{-1}f|\beta_{j}-\chi'(\gamma)^{-1}\sum_{j}c_{\chi, \chi'}(\beta_{j})^{-1}f|\beta_{j}\gamma \\
&=\sum_{j}c_{\chi, \chi'}(\beta_{j})^{-1}f|\beta_{j}-\chi'(\gamma)^{-1}\sum_{j}c_{\chi, \chi'}(\beta_{j})^{-1}f|\gamma_{j}\beta_{\sigma_{\gamma}(j)} \\
&=\sum_{j}c_{\chi, \chi'}(\beta_{j})^{-1}f|\beta_{j}-\chi'(\gamma)^{-1}\sum_{j}c_{\chi, \chi'}(\beta_{j})^{-1}(\chi(\gamma_{j})f-\chi(\gamma_{j})h_{\gamma_{j}})|\beta_{\sigma_{\gamma}(j)} \\
&=\sum_{j}c_{\chi, \chi'}(\beta_{j})^{-1}f|\beta_{j}-\sum_{j}c_{\chi, \chi'}(\beta_{\sigma_{\gamma}(j)})^{-1}f|\beta_{\sigma_{\gamma}(j)}+\sum_{j}c_{\chi, \chi'}(\beta_{\sigma_{\gamma}(j)})^{-1}h_{\gamma_{j}}|\beta_{\sigma_{\gamma}(j)} \\
&=\sum_{j}c_{\chi, \chi'}(\beta_{\sigma_{\gamma}(j)})^{-1}h_{\gamma_{j}}|\beta_{\sigma_{\gamma}(j)}
\end{align*}
The last term is a finite sum of smooth functions so it is also smooth function on $\rr$ itself (except finitely many points). \qed
\end{proof}

So, when are $\chi$ and $\chi'$ compatible at $\alpha$? To check compatiblity, we will use the following lemma. 
\begin{lemma}
\label{compcheck}
$\chi$ and $\chi'$ are compatible at $\alpha$ if and only if $\chi(\gamma) =\chi'(\alpha^{-1}\gamma\alpha)$ for any $\gamma\in \Gamma\cap \alpha\Gamma\alpha^{-1}$.
\end{lemma} 
\begin{proof}
($\Rightarrow$) Let $\gamma \in\Gamma\cap\alpha\Gamma\alpha^{-1}$, so that $\gamma = \alpha\gamma'\alpha^{-1}\Leftrightarrow \gamma\alpha = \alpha\gamma'$ for some $\gamma'\in \Gamma$. Then by compatibility, we have $\chi(\gamma) = c_{\chi, \chi'}(\gamma\alpha) = c_{\chi, \chi'}(\alpha\gamma') = \chi'(\gamma') = \chi'(\alpha^{-1}\gamma\alpha)$. \\
($\Leftarrow$) Assume that $\chi(\gamma)=\chi'(\alpha^{-1}\gamma\alpha)$ holds for any $\gamma\in \Gamma\cap \alpha\Gamma\alpha^{-1}$. If  $\gamma_{1}\alpha\gamma_{2} = \delta_{1}\alpha\delta_{2}\Leftrightarrow \gamma_{1}^{-1}\delta_{1}\alpha=\alpha\gamma_{2}\delta_{2}^{-1}$ in $\Gamma\alpha\Gamma$, then 
$$
\chi(\gamma_{1})^{-1}\chi(\delta_{1})=\chi(\gamma_{1}^{-1}\delta_{1}) = \chi(\alpha\gamma_{2}\delta_{2}^{-1}\alpha^{-1}) = \chi'(\gamma_{2}\delta_{2}^{-1}) = \chi'(\gamma_{2})\chi'(\delta_{2})^{-1}.
$$
so $\chi(\gamma_{1})\chi'(\gamma_{2}) = \chi(\delta_{1})\chi'(\delta_{2})$. 
\qed
\end{proof}
We will focus on the case when  $\Gamma=\Gamma_{0}(N)$ and $\alpha=\left(\begin{smallmatrix}1&0\\0&p\end{smallmatrix}\right)$, for $p\nmid N$. 
We will use a notation $T_{p, \chi, \chi'}$ for $T_{\alpha,\chi, \chi'}$. 
In this case, it is known that (see \cite{ds05}) the set of representatives of orbits $\Gamma_{0}(N)\backslash \Gamma_{0}(N)\left(\begin{smallmatrix}1&0\\0&p\end{smallmatrix}\right)\Gamma_{0}(N)$ can be chosen as 
$$
\beta_{0}=\begin{pmatrix} 1 & 0 \\ 0 & p\end{pmatrix}, \beta_{1}=\begin{pmatrix} 1& 1 \\ 0 & p \end{pmatrix}, \cdots, \beta_{p-1}=\begin{pmatrix} 1& p-1 \\ 0 &p\end{pmatrix}, \beta_{\infty}=\begin{pmatrix} p & 0 \\ 0 & 1\end{pmatrix}.
$$
%If both $\chi$ and $\chi'$ are trivial, then they are compatible at (any) $\alpha$ and  $T_{p, 1, 1}$ is same as ordinary Hecke operator $T_{p}$ on the space of modular forms. 

As an example, we will consider the Zagier's quantum modular form $f:\qq\to \cc$ introduced in  the Section 2.  
Recall that the function $f$ is the weight 1 quantum modular form on $\Gamma_{0}(2)$ with the multiplier system $\chi$ defined as
$$
\chi(T)=\chi\left(\m{1&1\\0&1}\right)=\zeta_{24},\qquad \chi(R)=\chi\left(\m{1&0\\2&1}\right)=\zeta_{24}.
$$
For any prime $p\geq 5$, we can easily check that $\Gamma\cap \alpha\Gamma\alpha^{-1} =\Gamma_{0}(2p)$ when $\Gamma = \Gamma_{0}(2)$ and $\alpha = \left(\begin{smallmatrix} 1&0\\0&p\end{smallmatrix}\right)$. By SAGE, we made a program to compute the value of $\chi(\gamma)$ for any given $\gamma\in \Gamma_{0}(2)$. Also, SAGE provides enviroment to find generators of congruence subgroups, so we've checked that $\chi$ and $\chi'=\chi^{p}$ are compatible at $p$ for any prime $5\leq p\leq 757$, i.e. $\chi(\gamma) = \chi(\alpha^{-1}\gamma\alpha)^{p}$ for any $\gamma\in \Gamma_{0}(2p)$ by compute its values on generators of $\Gamma_{0}(2p)$. 
For example, in case of $p=5$, generators of $\Gamma_{0}(10)$ are 
$$
\begin{pmatrix} 1&1\\0&1\end{pmatrix}, \quad \begin{pmatrix} 3&-1 \\ 10&-3\end{pmatrix}, \quad \begin{pmatrix} 19 & -7 \\ 30 & -11\end{pmatrix}, \quad \begin{pmatrix} 11 & -5 \\ 20 & -9 \end{pmatrix}, \quad  \begin{pmatrix} 7 & -5 \\ 10 & -7 \end{pmatrix} 
$$ 
and both $\chi\left(\left(\begin{smallmatrix} a&b\\c&d \end{smallmatrix} \right)\right)$ and $\chi\left(\alpha\left(\begin{smallmatrix} a & b\\c&d\end{smallmatrix}\right)\alpha^{-1}\right)^{5} = \chi\left(\left(\begin{smallmatrix} a& 5b \\ c/5 & d \end{smallmatrix}\right)\right)^{5}$ have values
$$
\zeta_{24}, \quad 1, \quad \zeta_{24}^{20}, \quad \zeta_{24}^{19}, \quad 1. 
$$

\begin{comment}

Now let $\chi'=\chi^{7}$. 
\begin{lemma}
$\chi$ and $\chi'$ are compatible at 7. 
\end{lemma}

\begin{proof}
We will define $c_{\chi, \chi'}:\Gamma_{0}(2)\sm{1&0\\0&7}\Gamma_{0}(2)\to \mathbb{S}^{1}$ as
\begin{align*}
c_{\chi, \chi'}(\beta_{\infty})&=c_{\chi, \chi'}(-T^{4}R^{-1}\sm{1&0\\0&7}T^{4}R^{-1})=1  \\
c_{\chi, \chi'}(\beta_{j})&=c_{\chi, \chi'}\left(\sm{1&0\\0&7}T^{j}\right)=\zeta_{24}^{7j}, \qquad 0\leq j\leq 6 
\end{align*}
To check well-definedness of $c_{\chi, \chi'}$, we only have to check for the following 16 elements:
\begin{alignat*}{4}
&\beta_{\infty}T=T^{7}\beta_{\infty}, \,\, &&\beta_{0}T=\beta_{1} , \,\, &&\beta_{1}T=\beta_{2}, \,\, && \beta_{2}T=\beta_{3},\\
&\beta_{3}T=\beta_{4},\,\,&&\beta_{4}T=\beta_{5},\,\,&&\beta_{5}T=\beta_{6},\,\,&&\beta_{6}T=T\beta_{0},\\
 &\beta_{\infty}R= -T^{4}R^{-1}\beta_{4},\,\,&&\beta_{0}R= R^{7}\beta_{0},\,\,&&\beta_{1}R= -R^{2}T^{2}T^{-1}\beta_{5},\,\,&&\beta_{2}R= RTR^{2}T^{-1}\beta_{6}, \\ 
&\beta_{3}R= RT^{3}\beta_{\infty}, \,\,&&\beta_{4}R= RT^{-2}R^{-2}\beta_{2}, \,\, &&\beta_{5}R= -TR^{-2}T^{-1}R^{-1}\beta_{3}, \,\,&& \beta_{6}R= -TR^{-7}\beta_{1}. 
\end{alignat*}
We can check that $c_{\chi, \chi'}$ is well defined for these elements, so $\chi$ and $\chi'$ are compatible at $7$. \qed
\end{proof}
\end{comment}

By compatibility of multiplier systems, we can define the Hecke operators $T_{p}^{\infty}=T^{\infty}_{p, \chi, \chi^{p}}:\QQ_{1}(\Gamma_{0}(2), \chi)\to \QQ_{1}(\Gamma_{0}(2), \chi')$ for $5\leq p \leq 757$. From 
\begin{align*}
\beta_{\infty} = \begin{pmatrix} p&0\\0&1\end{pmatrix} = -T^{\frac{p+1}{2}}R^{-1} \begin{pmatrix}1&0\\0&p\end{pmatrix}T^{\frac{p+1}{2}}R^{-1}, \quad \beta_{j} = \begin{pmatrix} 1 & j \\ 0 &p \end{pmatrix} = \begin{pmatrix} 1 & 0 \\ 0 & p \end{pmatrix} T^{j},
\end{align*}
we have
$$
c_{\chi, \chi^{p}}(\beta_{\infty}) = \zeta_{24}^{\frac{p+1}{2}}\zeta_{24}^{-1}(\zeta_{24}^{\frac{p+1}{2}}\zeta_{24}^{-1})^{p} = \zeta_{24}^{\frac{p^{2}-1}{2}} = (-1)^{\frac{p^{2}-1}{24}}, \quad c_{\chi, \chi^{p}}(\beta_{j}) = \zeta_{24}^{pj}. 
$$
(Note that $24|p^{2}-1$ for any $p\geq 5$.) Now we get our main result.
\begin{theorem}
Let
$$
g(x)=T_{p}^{\infty}f(x)=(-1)^{\frac{p^{2}-1}{24}}f(px)+\frac{1}{p}\sum_{j=0}^{p-1}\zeta_{24}^{-pj}f\left(\frac{x+j}{p}\right).
$$
Then $g$ should satisfy 
$$
g(x+1)=\zeta_{24}^{p}g(x), \qquad \frac{1}{|2x+1|}g\left(\frac{x}{2x+1}\right)=\zeta_{24}^{p}g(x)+H(x)
$$
where $H$ can be extended smoothly on $\rr$ except finitely many points. 
\end{theorem}

We can naturally come up with the following questions. \\
\textbf{Question 1}. \emph{Are $\chi$ and $\chi^{p}$ compatible at $\alpha = \left(\begin{smallmatrix} 1&0 \\0&p\end{smallmatrix}\right)$ for any $p\geq 5$? i.e. Does the equation 
$$
\chi\left(\begin{pmatrix} a&b \\ c&d \end{pmatrix} \right) = \chi\left(\begin{pmatrix} a & pb \\ c/p & d \end{pmatrix} \right)^{p}
$$
hold for any $\gamma = \left(\begin{smallmatrix} a& b\\ c&d \end{smallmatrix}\right)\in \Gamma_{0}(2p)$? We may need an explicit formula of $\chi(\gamma)$ in terms of $a, b, c, d$.}\\ \\
\textbf{Question 2}. \emph{Is there any Hecke {eigenform} in the space of quantum modular forms? For example, if $p\equiv 1\Mod{24}$, then $T_{p}^{\infty}:\QQ_{1}(\Gamma_{0}(2), \chi)\to \QQ_{1}(\Gamma_{0}(2), \chi^{p}) = \QQ_{1}(\Gamma_{0}(2), \chi)$. Is Zagier's quantum modular form $f$ an eigenform with respect to such $T_{p}^{\infty}$? i.e. is there $\lambda_{p}\in \cc$ such that $T_{p}^{\infty}f(x) = \lambda_{p} f(x)$ for any $x\in \qq$?} \\ 
Actually, we will prove that this is true in a subsequent paper. \\ \\
\textbf{Question 3}. \emph{Can we extend the result for half-integral weight (quantum) modular forms?} \\
In \cite{za10}, there are a lot of examples of quantum modular forms of half-integral weight. If we can develop the similar theory that can be applied these forms, we might get some interesting results.

\end{document}